\documentclass[11pt,oneside,english]{amsart}
\usepackage[T1]{fontenc}
\usepackage[latin9]{inputenc}
\usepackage{babel}
\usepackage{verbatim}
\usepackage{amstext}
\usepackage{amsthm}
\usepackage{amssymb}
\usepackage[unicode=true,pdfusetitle,
 bookmarks=true,bookmarksnumbered=false,bookmarksopen=false,
 breaklinks=false,pdfborder={0 0 1},backref=false,colorlinks=false]
 {hyperref}

\makeatletter
\numberwithin{equation}{section}
\numberwithin{figure}{section}
\theoremstyle{plain}
\newtheorem{thm}{\protect\theoremname}
  \theoremstyle{remark}
  \newtheorem{rem}[thm]{\protect\remarkname}
  \theoremstyle{plain}
  \newtheorem{lem}[thm]{\protect\lemmaname}
  \theoremstyle{plain}
  \newtheorem{cor}[thm]{\protect\corollaryname}
  \theoremstyle{plain}
  \newtheorem*{thm*}{\protect\theoremname}

\newcommand{\ra}{\rightarrow}

\newcommand{\cG}{{\mathcal G}}

\newcommand{\cO}{{\mathcal O}}

\newcommand{\bN}{{\mathbb N}}

\newcommand{\bR}{{\mathbb R}}

\newcommand{\bZ}{{\mathbb Z}}
\newcommand{\bQ}{{\mathbb Q}}
\newcommand{\bF}{{\mathbb F}}

\newcommand{\bP}{{\mathbb P}}

\newcommand{\set}[1]{\left\{#1\right\}}
\newcommand{\pa}[1]{\left(#1\right)}
\newcommand{\av}[1]{\left|#1\right|}

\newif\ifdraft\drafttrue

\makeatother

  \providecommand{\corollaryname}{Corollary}
  \providecommand{\lemmaname}{Lemma}
  \providecommand{\remarkname}{Remark}
  \providecommand{\theoremname}{Theorem}
\providecommand{\theoremname}{Theorem}

\begin{document}

\title{Continued fractions of arithmetic sequences of quadratics}

\author{Menny Aka}
\begin{abstract}
Let $x$ be a quadratic irrational and let $\bP$ be the set of prime
numbers. We show the existence of an infinite set $S\subset\bP$ such
that the statistics of the period of the continued fraction expansions
along the sequence $\set{px:p\in S}$ approach the `normal' statistics
given by the Gauss-Kuzmin measure. Under the generalized Riemann hypothesis,
we prove that there exist full density subsets $S\subset\bP$ and
$T\subset\bN$ satisfying the same assertion. We give a rate of convergence
in all cases. 
\end{abstract}

\address{ETH Zürich\\
Rämistrasse 101\\
CH-8092 Zürich\\
Switzerland }

\email{mennyaka@math.ethz.ch}
\maketitle

\section{Introduction}

In this paper we address some natural questions that were raised in
\cite{AS1} on the study of the evolution of the continued fraction
expansion within a fixed real quadratic field. The reader is referred
there for more details and motivation. We review here the necessary
definitions for the presentation of our results.

For a real number $x\in[0,1]\setminus\bQ$ let $x=[a_{1}(x),a_{2}(x),\dots]$
be its uniquely defined continued fraction expansion. For a finite
sequence of natural numbers $w=(w_{1},\dots,w_{k})$ we define the
\textit{frequency of appearance} of the pattern $w$ in the continued
fraction expansion of $x$, by
\[
D(x,w)=\lim_{N}\frac{1}{N}\#\set{1\le n\le N:w=(a_{n+1}(x),\dots,a_{n+k}(x))}.
\]
We make the convention that for $x\in\bR\setminus\bQ$, $D(x,w):=D(\bar{x},w)$
where $\bar{x}\in[0,1]$ and $\bar{x}=x\mod1$. As we explain in the
introduction of \cite{AS1}, for Lebesgue almost any $x$, $D(x,w)$
exists and is equal to some explicit integral which depends only on
$w$. We denote this almost sure value as $c_{w}$.

In \cite{AS1} we fix a quadratic irrational $x$ and study the behavior
of $D(p^{n}x,w)$ as $n\ra\infty$ where $p$ is a \emph{fixed} natural
number, or more generally, the behavior of $D(p_{n}x,w)$ when $p_{n}\ra\infty$
are supported on a fixed finite set of primes $S$ (by supported we
mean that if a prime $p$ divides $p_{n}$ for some $n$ then $p\in S$).
The methods of \cite{AS1} fail when the numbers $p_{n}$ are not
supported of finite set of primes. The aim of this note is merely
to rephrase known results in analytic number theory using a refinement
of Duke's Theorem (See section \ref{subsec:Duke's-Theorem-for subcolelction})
in order to answer question (4) of \cite[\S 2.7]{AS1} and some related
results. Using the results of \cite{Cohen2007,Nark88} we have the
following unconditional result: 
\begin{thm}
\label{thm:unconditional}There exist at most two real quadratic fields
such that if a quadratic irrational $x$ does not belong to them,
there exist an infinite subset $S\subset\bP$ and $\delta>0$\textup{
}such that for any pattern $w$ there exists\textup{ a constant $C=C(w,x)$}
with 
\[
p\in S\implies\av{D(px,w)-c_{w}}<Cp^{-\delta}.
\]
\end{thm}

Using \cite{Par2003} with with \cite{MH2006,Popa2006}, we have the
following two conditional results:
\begin{thm}
\label{thm:GRH on Natural numbers and primes}Let $x$ be a quadratic
irrational. Then, assuming that the Generalized Riemann Hypothesis
(GRH) holds, there exist a density one subset $S\subset\bN$ and $\delta>0$
such that for any pattern $w$\textup{ there exists a constant $C=C(w,x)$}
with 
\[
m\in S\implies\av{D(mx,w)-c_{w}}<Cm^{-\delta}.
\]
Similarly, let $P$ denote the set of primes. Assuming that the GRH
holds, there exist a density one sequence $S\subset P$ and $\delta>0$
such that for any pattern $w$\textup{ there exists a constant $C=C(w,x)$}
with
\[
p\in S\implies\av{D(px,w)-c_{w}}<Cp^{-\delta}.
\]
\end{thm}

\begin{rem}
Using the same methods, Theorems \ref{thm:unconditional} and \ref{thm:GRH on Natural numbers and primes}
can be slightly generalized. For example, these theorems remain valid
if we can consider $D(w,\frac{x}{m})$ instead of $D(w,mx)$ for $m$
coprime with ${\rm disc(\bQ(x))}$ (as the associated orders are the
same, i.e. $\cO_{mx}=\cO_{\frac{x}{m}}$, see $\S$\ref{subsec:Duke's-Theorem-for subcolelction}).
Theorems \ref{thm:unconditional} and \ref{thm:GRH on Natural numbers and primes}
should be viewed as what we consider to be representative results
that can be achieved via these techniques.
\end{rem}

\subsection{Organization of this note}

As mentioned above the following note is merely a reformulation of
known results in analytic number theory. As in \cite{AS1} the above
results follow from equidistribution of geodesic loops on the modular
surface. In this context, there is an arithmetically defined sequence
of collections $\cG_{d}$ of closed geodesics that equidistribute
as $d\ra\infty$; this statement is now classically known as \textquotedbl{}Duke's
Theorem\textquotedbl{}. When one considers the statistics of continued
fraction sequence along some arithmetic sequences, one is led to consider
subcollections (in fact we will consider one ``long'' geodesic)
of $\cG_{d}$. When these subcollections occupy a substantial part
of $\cG_{d}$, one can hope that they will equidistribute too (c.f.
the far-reaching conjecture \cite[Conjecture 1.11]{ELMV1}). Section
\ref{sec:Preliminaries} is devoted to describe these notions and
state a variant of Duke's Theorem for subcollections. This establishes
the bridge between statistics of continued fractions expansion and
closed geodesics. Section \ref{sec:Proofs} starts with the simple
observations connecting the later to orders of certain matrices. Then
we state several results from analytic number theory about these orders
(which can be viewed as analogues of Artin's conjecture on primitive
roots) and deduce our results. 

\subsection*{Acknowledgments}

This work emerged from a joint work with Uri Shapira and I wish to
thank him for many discussions on this topic and his interests in
this project. We would like to thank Lior Rosenzweig to referring
him to Par Kurlberg's work and to Paul Nelson and Philippe Michel
for fruitful discussions. Many thanks are due to Andreas Wieser for
a careful reading of this note, and to Alain Valette for encouraging
the publication of this manuscript. Thanks are due to the anonymous
referee for improving the exposition. We acknowledge the support of
Advanced Research Grant 228304 from the European Research Council. 

\section{Preliminaries\label{sec:Preliminaries}}

We freely use notions from algebraic number theory in the context
of quadratic real extensions, as orders, discriminant, regulators,
etc. The reader is referred to \cite[Page 133]{cox2011primes} or
any algebra number theory book for background on these notions.

As in \cite{AS1}, the above theorem follows from equidistribution
theorems of certain geodesic loops on the modular surface so we briefly
record the relevant settings. For the geodesic flow on the modular
surface and also for its relation to continued fractions, we refer
the reader to \cite[Section 9.6]{EW10}. Let $X_{2}={\rm {\rm SL_{2}}(\bZ)\setminus SL_{2}}(\bR)=:\Gamma\setminus G$
be the unit tangent bundle of the modular surface. The geodesic flow
on $X_{2}$ is given by the action of the diagonal group 
\[
A:=\set{a_{t}:=\pa{\begin{array}{cc}
e^{\frac{-t}{2}} & 0\\
0 & e^{\frac{t}{2}}
\end{array}}}_{t\in\mathbb{R}}<G
\]
 sending $\Gamma g$ to $\Gamma ga_{t}^{-1}$. An element $x\in\mathbb{R}$
is called\emph{ quadratic irrational} if $\bQ(x)$ is a real quadratic
field. For a quadratic irrational $x,$ we let 
\[
g_{x}=\begin{cases}
\frac{1}{\sqrt{x-\bar{x}}}\pa{\begin{array}{cc}
x & \bar{x}\\
1 & 1
\end{array}} & \text{ if }x-\bar{x}>0\\
\frac{1}{\sqrt{\bar{x}-x}}\pa{\begin{array}{cc}
\bar{x} & x\\
1 & 1
\end{array}} & \text{ if }x-\bar{x}<0
\end{cases}
\]
and set $L_{x}=\Gamma g_{x}A$. It is known that it is a closed geodesic
in $X_{2}$ (see \cite[Lemma 8.2]{AS1} ). 

\subsection{Duke's Theorem for subcollections\label{subsec:Duke's-Theorem-for subcolelction}}

Proofs of all the statement of this subsection can be found in \cite[\S 2]{ELMV}
and the references within). For any closed geodesic $L$ one can choose
a quadratic irrational $x$ such that $L=L_{x}$. We can associate
with $L_{x}$ an order in $\bQ(x)$ by $\cO_{x}:=\set{a:a\Lambda_{x}\subseteq\Lambda_{x}}$
where $\Lambda_{x}:=\bZ\cdot1\oplus\bZ\cdot x$. This order does not
depend on the choice of $x$ and therefore for any geodesic loop $L=L_{x}$
we can associate a discriminant by ${\rm Disc}(L)={\rm Disc}(\cO_{x})$.
Let $\cG_{D}$ be the set of geodesic loops which have discriminant
$D$. 
\begin{thm}
\label{thm:Fact on duke ingridents}Let $D$ be positive non-square
discriminant and $\cO_{D}:=\bZ[\frac{D+\sqrt{D}}{2}]$ be the order
of discriminant $D$. We have:
\end{thm}

\begin{enumerate}
\item $\av{\cG_{D}}=\av{{\rm Pic}(\cO_{D})}$ where ${\rm Pic}(\cO_{D})$
is the ideal class group of $\cO_{D}$.
\item The length of any $L$ in $\cG_{D}$ is equal to ${\rm Reg}(\cO_{D})$,
the regulator of $\cO_{D}$.
\item The total length of the collection $\cG_{D}$ is ${\rm Reg}(\cO_{D})\av{\cG_{D}}=D^{\frac{1}{2}+o(1)}$\label{enum: total length}.
\end{enumerate}
\begin{proof}
See \cite[ \S2]{ELMV}.
\end{proof}
For any $L\in\cG_{D}$ we let $\mu_{L}$ be the normalized length
measure supported on $L$ and we let $l(L)$ denote the length of
$L$. By Theorem \ref{thm:Fact on duke ingridents}, we have for $L\in\cG_{D}$
that $l(L)={\rm Reg}(\cO_{D})$.
\begin{thm}
\label{thm:subcollection Duke}There exist a $\delta>0$ with the
following property: Let $\set{L_{i}}$ be a sequence of geodesics
with $l(L_{i})={\rm Reg}(L_{i})>{\rm Disc}(L_{i})^{\frac{1}{2}-a}$
for some $0\leq a<\delta$. Then for any $f\in C_{c}^{\infty}(X)$
we have 
\[
\av{\mu_{L_{i}}(f)-\mu_{X_{2}}(f)}\ll S(f){\rm Disc}(L_{i})^{a-\delta}
\]
where $S(f)$ is a Sobolev norm.
\end{thm}

\begin{proof}
When $\set{{\rm Disc}(L_{i})}_{i=1}^{\infty}$ are fundamental discriminants
(i.e. square-free) this theorem follows from \cite[Theorem 6.5.1]{Popa2006}
or \cite[Theorem 6]{MH2006}. In \cite[Corollary 1.4]{HarcosThesis}
one finds an explicit exponent $\delta=\frac{1}{2827}$. For the results
of this paper, we are mainly in non-fundamental discriminants. Unfortunately,
such a result does not exist in print. Nevertheless it is knows to
experts (see e.g. \cite{MV-ICM}) and it will appear in an upcoming
monograph by Philippe Michel.  The reader is also referred to \cite{SubcollectionDuke}
for a related result and general discussion of such results.
\end{proof}

\section{Proofs\label{sec:Proofs}}

Let $D>2$ be a fundamental discriminant, $\cO_{D}$ the unique order
of discriminant $D$ and $K={\rm Frac}(\cO_{D})$. Hereafter, we make
the explicit choice of $x_{D}=\sqrt{D}$ when $D\neq1({\rm mod}\,\,4)$
and $x_{D}=\frac{1+\sqrt{D}}{2}$ when $D=1({\rm mod}\,\,4)$ so $\bZ[x_{D}]=\cO_{D}$.
\begin{lem}
\label{lem:disc of Z=00005BNx=00005D}For any integer $N\in\bN$,
the discriminant of the order $\bZ[Nx_{D}]$ is $N^{2}D$.
\end{lem}

\begin{proof}
It is well known that for any order $\cO\subset\cO_{K}$, 
\[
{\rm Disc}(\cO)={\rm Disc}(\cO_{K})(\cO_{K}:\cO)^{2}.
\]
As $[\bZ[x_{D}]:\bZ[Nx_{D}]]=N$ the lemma follows. 
\end{proof}
As there is a unique order of a given (not necessarily fundamental)
discriminant, we conveniently denote $\cO_{N^{2}D}:=\bZ[Nx_{D}]$
the order of discriminant $N^{2}D$.

\subsection{Identifying $\cO_{D}$ with $\bZ^{2}$ .}

Using the basis $B=\{1,x_{D}\}$ of $K$ we identify $K$ with $\bQ^{2}$
and $\cO_{D}$ with $\bZ^{2}$. For $\alpha\in K$ we write $m_{\alpha}$
as the multiplication map by $\alpha$ on $K$. Writing $m_{\alpha}$
in the basis $B$ (using rows), we get an embedding $\varphi:K\ra M_{2}(\bQ)$.
The reader can readily verify that $\varphi^{-1}(M_{2}(\bZ))=\cO_{D}$,
and that for all $\alpha\in K$, we have ${\rm Trace_{K/\bQ}}(\alpha)={\rm Trace}(\varphi(\alpha)),{\rm Norm_{K/\bQ}}(\alpha)={\rm Det}(\varphi(\alpha))$.
\begin{lem}
\label{lem:being in N=0002C62D}For any $\alpha\in\cO_{D}$ we have
$\alpha\in\cO_{N^{2}D}$ if and only if $\varphi(\alpha)$ is a scalar
matrix modulo $N$.
\end{lem}

\begin{proof}
By definition $\varphi(\alpha)=\pa{\begin{matrix}P & Q\\
R & S
\end{matrix}}$ where $\alpha=P+Qx_{D},\alpha x_{D}=R+Sx_{D}$. Thus if $\varphi(\alpha)$
is a scalar matrix modulo $N$, in particular $Q=0({\rm mod}\,\,N)$,
so $\alpha\in\cO_{N^{2}D}$. 

Now, let $\alpha\in\cO_{N^{2}D}=\bZ[Nx_{D}]$, and write $\alpha=a+bNx_{D}$
for some $a,b\in\bZ$. As $\varphi$ is a ring-homomorphism, we have
\[
\varphi(\alpha)=aI_{2}+bN\varphi(x_{D})
\]
where $I_{2}$ is the identity matrix. As $x_{D}$ is integral, $\varphi(x_{D})\in M_{2}(\bZ)$
and therefore $\varphi(\alpha)=aI_{2}$ modulo $N$.  
\end{proof}
\begin{cor}
\label{cor: reg is bounded below by order}Let $\epsilon=\epsilon_{D}$
be a fundamental unit of $\cO_{D}$ and for any integer $N\in\bN$,
let $R(N)=\min\{k:\varphi(\epsilon_{D})^{k}=\pm I\text{ modulo N}\}$.
We have 
\[
{\rm Reg}(\cO_{N^{2}D})={\rm Reg}(\cO_{D})R(N).
\]
It follows that \textup{$Reg(\cO_{N^{2}D})\gg{\rm ord_{N}(\varphi(\epsilon))}$
}where\textup{ ${\rm ord_{N}(\varphi(\epsilon))}$ }is the order of
the reduction mod $N$ of $\varphi(\epsilon)$ in ${\rm SL_{2}}(\bZ/N\bZ)$. 
\end{cor}

\begin{proof}
By definition of the regulator, we have that 
\[
{\rm Reg}(\cO_{N^{2}D})={\rm Reg}(\cO_{D})\cdot(\cO_{D}^{\times}:\cO_{N^{2}D}^{\times}),
\]
so we need to show that 
\begin{equation}
(\cO_{D}^{\times}:\cO_{N^{2}D}^{\times})=R(N).\label{eq:R(N) equality}
\end{equation}
Using Lemma \ref{lem:being in N=0002C62D} we have that $\epsilon^{k}\in\cO_{N^{2}D}$
if and only if $\varphi(\epsilon^{k})$ is a scalar matrix mod $N$,
and since $\det(\varphi(\epsilon))={\rm Norm}(\epsilon)=1$, this
is if and only if $\varphi(\epsilon^{k})^{-1}=\varphi(\epsilon^{-k})$
is a scalar matrix mod $N$, which is equivalent to $\epsilon^{-k}\in\cO_{N^{2}D}$.
We infer that $\epsilon^{k}\in\cO_{N^{2}D}$ implies that $\epsilon^{k}\in\cO_{N^{2}D}^{\times}$.
After noting that $\det(\varphi(\epsilon))={\rm Norm}(\epsilon)=\pm1$
implies that whenever $\varphi(\epsilon^{l})$ is a scalar matrix
mod $N$ for some $l,$ it is in fact $\pm I$ mod $N$, (\ref{eq:R(N) equality})
follows, and therefore also the first part of the corollary. 

It is clear that $R(N)$, up to possibly a factor of $2$, is ${\rm ord_{N}(\varphi(\epsilon))}$
so the second part also follows. 
\end{proof}

\subsection{\label{subsec:realting orders}Studying \textmd{${\rm ord}_{N}(\varphi(\epsilon_{D}))$.}}

Let us review and fix some notations: As before, for a matrix $M\in{\rm SL}_{2}(\bZ)$,
${\rm ord}_{N}(M)$ is the order of the reduction mod $N$ of $M$
in ${\rm SL}{}_{2}(\bZ/N\bZ)$. Similarly, for $a\in\cO_{D}^{\times}$
we denote by ${\rm ord}_{N}(a)$ the order of $a$ in $(\cO_{D}/N\cO_{D})^{\times}$.
The following observation is very simple but crucial to us so we enunciate
it in a lemma:
\begin{lem}
\label{lem:martix order vs ring order}For all $N\in\bN$, \textup{${\rm ord}_{N}(\varphi(\epsilon_{D}))={\rm ord}_{N}(\epsilon_{D})$.}
\end{lem}

\begin{proof}
Note that $a\in\cO_{D}^{\times}$ represent the trivial element of
$(\cO_{D}/N\cO_{D})^{\times}$ if and only if $a=1\text{ mod }N\cO_{D}$.
We have $\epsilon_{D}^{k}=1\text{ mod }N\cO_{D}$ if and only if $\epsilon_{D}^{k}=1+N\alpha,\alpha\in\cO_{D}$
if and only if $\varphi(\epsilon_{D})^{k}=I+N\varphi(\alpha),\alpha\in\cO_{D}$
if and only if $\varphi(\epsilon_{D})^{k}=I\text{ mod }N$, where
the last \textquotedbl{}if\textquotedbl{} follows from the fact that
$\varphi$ is a field embedding of $K$ and since $\varphi^{-1}(M_{2}(\bZ))=\cO_{D}$.
\end{proof}
The study of ${\rm ord_{N}(\epsilon_{D})}$, at least for prime moduli,
can be considered as a quadratic analogue of Artin's conjecture on
primitive roots \cite{Artin65}. This conjecture asserts that any
integer $a\neq\pm1$ is a primitive root in $(\bZ/p\bZ)^{\times}$
for a set of primes of some explicitly given positive density. Under
the Generalized Riemann Hypothesis (abbreviated henceforth as GRH),
this conjecture has been verified by Hooley \cite{Hoo67}. Unconditionally,
with some restrictions (analogous to the exceptions in Theorem \ref{thm:Narkiewic})
Gupta and Murty \cite{GM84} and Heath-Brown's \cite{HB86} find an
infinite set of primes for which $a$ is primitive root modulo these
primes. All the results below are various analogues of these works
and their proofs are very similar in spirit. 

Although we are not using it in the sequel, it is important for us
to mention the following quadratic analogue of Hooley's result \cite{Hoo67}
on Artin's conjecture that was obtained by Roskam \cite{Ros2000}:
\begin{thm*}
Let $P$ denote the set of prime numbers. Under the GRH, there exists
a subset $S=S(D)\subset P$ of positive (explicitly given) density
with ${\rm ord}_{p}(\epsilon_{D})$ \textup{is maximal amongst all
elements of $(\cO_{D}/p\cO_{D})^{\times}$ whenever }$p\in S$. In
particular, 
\[
p\in S\implies{\rm ord}{}_{p}(\epsilon_{D})\gg p.
\]
\end{thm*}

Narkiewic \cite{Nark88} proved a quadratic analogue of the results
of Gupta and Murty \cite{GM84} and Heath-Brown's \cite{HB86} on
Artin's Conjecture:
\begin{thm}
\label{thm:Narkiewic}Unconditionally, with the possible exception
of two discriminants $D$, there exists an infinite subset $S=S(D)$
of the primes numbers with ${\rm ord}_{p}(\epsilon_{D})$ \textup{is
maximal amongst all elements of $(\cO_{D}/p\cO_{D})^{\times}$ whenever
}$p\in S$. In particular, 
\[
p\in S\implies{\rm ord}{}_{p}(\epsilon_{D})\gg p.
\]
\end{thm}

\begin{rem}
In order to have a better understanding of the bound $\gg p$, it
is informative to understand the maximal possible order that $\epsilon_{D}$
can have in $(\cO_{D}/p\cO_{D})^{\times}$. Let $p$ be a rational
odd prime which is unramified in $\cO_{D}$.  If $p$ splits in $D$
then $\cO_{D}/p\cO_{D}\cong\bF_{p}\times\bF_{p}$ so $(\cO_{D}/p\cO_{D})^{\times}\cong\bF_{p}^{\times}\times\bF_{p}^{\times}.$
So the maximal order of elements in $(\cO_{D}/p\cO_{D})^{\times}$
is $p-1$. When $p$ is inert in $\cO_{D}$, we have that $\cO_{D}/p\cO_{D}\cong\bF_{p^{2}}$.
If ${\rm Norm_{K/\bQ}}(\epsilon_{D})=1$ then the image of $\epsilon_{D}$
belongs to the kernel of the norm map from $\bF_{p^{2}}^{\times}\ra\bF_{p}^{\times}$
which is a cyclic group of order $\frac{p^{2}-1}{p-1}=p+1$. Similarly
if ${\rm Norm_{K/\bQ}}(\epsilon_{D})=-1$, then the image of $\epsilon_{D}$
belongs to a cyclic group of order $2(p+1)$. In all cases the maximal
possible order is $\gg p$. 

In fact, the set $S(D)$ of Theorem \ref{thm:Narkiewic} consists
of rational primes that split in $\bQ(\sqrt{D})$ so Theorem \ref{thm:Narkiewic}
is true as stated. Joseph Cohen \cite{Cohen2007} proved a similar
result for non-split primes under some more (rather mild) restrictions
on the possible discriminants that should be excluded.

Relaxing mildly the condition of being maximal, Par Kurlberg \cite{Par2003}
was able to prove a much stronger \emph{conditional} theorem:
\end{rem}

\begin{thm}
\label{thm: Par} Assume the GRH holds and let $M$ be a hyperbolic
matrix in ${\rm SL_{2}}(\bZ)$. Then, for every $\epsilon>0$, there
exists a subset $S\subset\bN$ of density one such that 
\[
N\in S\implies{\rm ord}{}_{N}(M)\gg_{\epsilon}N^{1-\epsilon}.
\]
Similarly, for every $\epsilon>0$, there exists a subset $T\subset{\rm Primes}$
of density one such that 
\[
p\in T\implies{\rm ord}{}_{p}(M)\gg_{\epsilon}p^{1-\epsilon}.
\]
\end{thm}

\subsubsection{Hecke neighbors}

The results of the previous section deal with properties of $x_{D}$
(and correspondingly with $\epsilon_{D}$), but we are interested
in a general element $x\in K$ which is not necessarily $x_{D}$.
To this end we first show that any two elements $K$ can be connected
to each other using the Hecke correspondence and then that this influences
the regulator and the discriminant only by a multiplicative factor.

Let $p$ be a rational prime; two elements $x,y\in K\setminus\bQ$
are called $p$-Hecke neighbors if $\Lambda_{x}\subset\Lambda_{y}$
or $\Lambda_{y}\subset\Lambda_{x}$ with index $p$. 
\begin{lem}
Let $x,y\in K\setminus\bQ$. Then $y$ can be connected to $x$ via
a chain of Hecke neighbors. This is, there exist $k\in\bN$ and 
\[
x_{0}=x,x_{1},\cdots,x_{k}=y
\]
such that $x_{i}$ and $x_{i+1}$ are $p_{i}$-Hecke neighbors, with
$p_{i}=p_{i}(x,y)$ some rational primes.
\end{lem}

\begin{proof}
Write $y=\frac{Ax+B}{C}$ with $A,B,C\in\bZ$. Factoring $A$ into
primes, we see that $x$ can be connected to $Ax$. Clearly $\Lambda_{Ax}=\Lambda_{Ax+B}$,
and by factoring $C$ into primes we see that $Ax+B$ can be connected
to $\frac{Ax+B}{C}=y$.
\end{proof}

\begin{lem}
\label{lem:Hecke neighbors}Let $x,y\in K\setminus\bQ$ be $p$-Hecke
neighbors. Then ${\rm Reg}(\cO_{x})\asymp_{p}{\rm Reg}(\cO_{y})$
and ${\rm Disc}(\cO_{x})\asymp_{p}{\rm Disc}(\cO_{y})$.
\end{lem}

\begin{proof}
Without less of generality we can assume that $\Lambda_{y}\subset\Lambda_{x}$.
Note that any element $\alpha\in K^{\times}$ act by multiplication
as an automorphism on $\pa{K,+}$. Thus any element $\alpha\in\cO_{x}^{\times}$
fixes $\Lambda_{x}$ and acts on its sublattices of index $p$. Therefore
there exists $1\leq\ell\leq p+1$ ($p+1$ being the number of index
$p$ sublattices) such that $\alpha^{\ell}$ fixes $\Lambda_{y}$,
i.e., $\alpha^{\ell}\in\cO_{y}^{\times}$. Vice versa, any element
$\alpha\in\cO_{y}^{\times}$ fixes $\Lambda_{y}$ and acts on its
suplattices (that is, lattices with $\Lambda\supset\Lambda_{y})$
of index $p$. As there are also $p+1$ such suplattices, there exists
$1\leq\tilde{\ell}\leq p+1$ such that $\alpha^{\tilde{\ell}}$ fixes
$\Lambda_{x}$, i.e., $\alpha^{\tilde{\ell}}\in\cO_{x}^{\times}$. 

Now, for any $z\in K\setminus\bQ$, 
\[
{\rm Reg}(\cO_{z})={\rm Reg}(\cO_{D})\cdot(\cO_{D}^{\times}:\cO_{z}^{\times}),
\]
so we have that ${\rm Reg}(\cO_{z})\asymp k(z)$ where $k(z)$ is
defined to be the integer such that $\cO_{z}^{\times}=\langle\pm\epsilon_{D}^{k(z)}\rangle$.
By the preceding discussion, $k(x)\asymp_{p}k(y)$ so the first part
of the lemma follows. 

For the second part, first recall (see e.g. \cite[Lemma 7.2]{cox2011primes})
that orders in quadratic fields are monogenic (i.e. of the form $\bZ[\alpha]$)
and since every order is contained in $\bZ[x_{D}]$ we can define
$l_{z}$ to be the integer with the property $\cO_{z}=\bZ[l_{z}x_{D}]$. 

Using this, the claim will follows once we show that $l_{x}\asymp_{p}l_{y}$.
To this end recall the definition of $\cO_{x}$ and note that 
\begin{equation}
l_{z}={\rm gcd}\set{m:mx_{D}\in\cO_{z}}.\label{eq:GCD}
\end{equation}
Still working under the assumption that $\Lambda_{y}\subset\Lambda_{x}$
and using the fact that any index $p$ sublattice of a lattice $\Lambda$
contains $p\Lambda$, we have:
\[
pl_{x}x_{D}\Lambda_{y}\subset pl_{x}x_{D}\Lambda_{x}\subset p\Lambda_{x}\subset\Lambda_{y}
\]
so (\ref{eq:GCD}) implies that $l_{y}|pl_{x}$ . Similarly, 
\[
pl_{y}x_{D}\Lambda_{x}=l_{y}x_{D}p\Lambda_{x}\subset l_{y}x_{D}\Lambda_{y}\subset\Lambda_{y}\subset\Lambda_{x}
\]
so $l_{x}|pl_{y}$. Therefore $\frac{1}{p}l_{y}\leq l_{x}\leq pl_{y}$.
As for any $l\in\bN$, ${\rm disc}(\bZ[lx_{D}])=l^{2}D$, we are done.

\end{proof}

\subsection{Proofs of Theorems \ref{thm:GRH on Natural numbers and primes} and
\ref{thm:unconditional}.}
\begin{proof}
The proofs of both theorems are very similar. We start with Theorem
\ref{thm:GRH on Natural numbers and primes}, and comment afterwards
on how to adapt the proof for Theorem \ref{thm:unconditional}. Let
$x$ be the quadratic irrational that appears in the statement of
the Theorem and let $D$ be the discriminant of $\bQ(x)$. We first
translate the statement on continued fractions to a statement on equidistribution
of closed geodesics on $X_{2}$: By \cite[Theorem 8.8, Theorem 8.9 and Corollary 2.10]{AS1}\footnote{Note that for different reasons the results in \cite{AS1} are phrased
for Lipschitz functions. For the purpose of this paper, one can easily
approximate characteristic functions on $[0,1]$ with smooth functions
with compact support, and translate the corresponding statement using
the results in \cite{AS1}.} it is enough to show that there exists a density one subset $S=\set{N_{i}}\subset\bN$
and $0<\delta_{1}$ such that for $f\in C_{c}^{\infty}(X)$ we have
\begin{equation}
\av{\mu_{L_{N_{i}x}}(f)-\mu_{X_{2}}(f)}\ll S(f){\rm Disc}(L_{N_{i}x})^{-\delta_{1}}.\label{eq:equi of loops}
\end{equation}
To this end, we will show that there exists a density one subset $S=\set{N_{i}}\subset\bN$
such that for every $\epsilon>0$ small enough 
\begin{equation}
{\rm Reg}(\cO_{N_{i}x})=l(L_{N_{i}x})\gg{\rm Disc}(L_{N_{i}x})^{\frac{1}{2}-\epsilon}.\label{eq: Reg disc condtion}
\end{equation}
This is indeed enough as by Theorem \ref{thm:subcollection Duke},
we have that (\ref{eq: Reg disc condtion}) implies (\ref{eq:equi of loops}). 

Toward establishing (\ref{eq: Reg disc condtion}), note that there
exist 
\[
x_{0}=x,x_{1},\cdots,x_{k}=x_{D}
\]
such that $x_{i}$ and $x_{i+1}$ are $p_{i}$-Hecke neighbors, with
$p_{i}=p_{i}(x)$ some rational primes. Recursive use of Lemma \ref{lem:Hecke neighbors}
implies that ${\rm Reg}(\cO_{x})\asymp_{c(x)}{\rm Reg}(\cO_{x_{D}})$
and ${\rm Disc}(\cO_{x})\asymp_{c(x)}{\rm Disc}(\cO_{x_{D}})$ where
$c(x)$ is a constant depending only on $x$. In fact, $c(x)$ is
proportional to the product of the rationals primes appearing in the
chain of Hecke neighbors connecting $x$ and $x_{D}$. Noting that
for every $N\geq1$, the same rational primes connect $Nx$ to $Nx_{D}$,
it follows that that ${\rm Reg}(\cO_{Nx})\asymp_{c(x)}{\rm Reg}(\cO_{Nx_{D}})$
and ${\rm Disc}(\cO_{Nx})\asymp_{c(x)}{\rm Disc}(\cO_{Nx_{D}})$,
that is, the proportionality constant $c(x)$ is independent of $N$.
It follows that it is enough to show (\ref{eq: Reg disc condtion})
for $x_{D}$ instead of $x$ so we can use the results of Subsection
\ref{subsec:realting orders} which implicitly deals with $x=x_{D}$. 

Assuming GRH and applying Theorem \ref{thm: Par} for the hyperbolic
matrix $\varphi(\epsilon_{D})$ and Corollary \ref{cor: reg is bounded below by order},
we find a density one subset $S=\set{N_{i}}\subset\bN$ such that
for every $\epsilon>0$ small enough we have 
\begin{equation}
{\rm Reg}(\cO_{N_{i}^{2}D})\gg{\rm ord_{N_{i}}(\varphi(\epsilon_{\text{D}}))}\gg_{\epsilon}N_{i}^{1-\epsilon}.\label{eq:reg bound}
\end{equation}
Note that $\Lambda_{Nx_{D}}=\bZ\cdot1\oplus\bZ\cdot Nx_{D}$ is in
fact the order $\bZ[Nx_{D}]$, so $\cO_{N_{i}x_{D}}=\bZ[N_{i}x_{D}]$
which has discriminant $N_{i}^{2}D$. Thus ${\rm Reg}(\cO_{L_{N_{i}x_{D}}})={\rm Reg}(\cO_{N_{i}^{2}D})$
and ${\rm Disc}(\cO_{L_{N_{i}x_{D}}})={\rm Disc}(\cO_{N_{i}^{2}D})=N_{i}^{2}D$.
Therefore (\ref{eq:reg bound}) implies (\ref{eq: Reg disc condtion})
for $x=x_{D}$ (where the implicit constants depend only on $x$,
which is fixed throughout). This concludes the proof of the Theorem
\ref{thm:GRH on Natural numbers and primes}.

The proof of the second part of Theorem \ref{thm:GRH on Natural numbers and primes}
follow the exact same lines using as an input the second part of Theorem
\ref{thm: Par}. 

We explain how to adapt the proof above to get a proof for Theorem
\ref{thm:unconditional}: Assuming $D$ does not fall into the 2 possible
exceptions, Theorem \ref{thm:Narkiewic} furnishes an infinite set
of primes $S=\set{p_{i}}$ with 
\[
{\rm ord}{}_{p_{i}}(\epsilon_{D})\gg p_{i}.
\]
Translating through Lemma \ref{lem:martix order vs ring order} and
using Corollary \ref{cor: reg is bounded below by order} we see that
\[
{\rm Reg}(\cO_{p_{i}^{2}D})\gg{\rm ord}_{p_{i}}(\varphi(\epsilon_{D}))\gg_{\epsilon}p_{i}
\]
which establish a (slightly stronger) analog of (\ref{eq:reg bound}).
One proceeds now exactly as above to establish Theorem \ref{thm:unconditional}.
\end{proof}

\bibliographystyle{plain}
\bibliography{AS2}

\end{document}